\documentclass{amsproc}
\usepackage[utf8]{inputenc}
  \usepackage{amsmath}
\usepackage{enumerate}
  \usepackage{amssymb}
  \usepackage{amsthm}
  \usepackage{bm}
  \usepackage{graphicx}
  \usepackage{color}
\usepackage{accents}
\usepackage{epsfig}
\usepackage{amsmath,amssymb}
\hoffset - 0.5 in
\newtheorem{Th}{Theorem}[section]
\newtheorem{lem}[Th]{Lemma}

\theoremstyle{definition}
\newtheorem{Def}[Th]{Definition}

\theoremstyle{remark}

\numberwithin{equation}{section}

\newcommand{\ds}{\displaystyle}

\newcommand{\printdate}{\today}

\title{On Flat Polynomials with Non-Negative Coefficients  }

\author{e. H. el Abdalaoui}
\address{Normandy University, University of Rouen
  Department of Mathematics, LMRS  UMR 60 85 CNRS\\
Avenue de l'Universit\'e, BP.12
76801 Saint Etienne du Rouvray - France .}
\email{elhoucein.elabdalaoui@univ-rouen.fr}
\urladdr{http://www.univ-rouen.fr/LMRS/Persopage/Elabdalaoui/}
\author{M. G. Nadkarni}
\address{Department of Mathematics, University of Mumbai, Vidyanagari, Kalina,  Mumbai, 400098, India}
\email{mgnadkarni@gmail.com}
\urladdr{http://insaindia.org/detail.php?id=N91-1080}


\dedicatory{}

\begin{document}
\maketitle
\begin{abstract} {We formulate and prove a necessary condition for a sequence of analytic trigonometric polynomials with real non-negative coefficients to be  flat a.e.}
\end{abstract}

\section{Introduction}\label{intro}  A sequence $P_j, j =1,2,\cdots$ of analytic trigonometric polynomials of $L^2$ norm one is said to be flat if the sequence $|P_j|,$  $j=1,2,\cdots$ of their absolute values
converges to the constant function 1 in some sense. The sense of convergence varies according to the situation. We will require that $P_j, j=1,2,\cdots$ converge in absolute value to the constant function 1 almost everywhere with respect to the Lebesgue measure on the unit circle. It is not known if such a flat sequence exists if we require that coefficients of each $P_j$ be  real and non-negative and uniformly bounded away from 1 over all j. The question is of interest since an affirmative answer to this question implies that there exists a invertible non-singular transformation on the unit interval with simple Lebesgue spectrum \cite{Abd-Nad} . Further if such a flat sequence $P_n, n=1,2,\cdots$ can be chosen from the class $B$ of polynomials of the type:
$$P(z) = \frac{1}{\sqrt m}(1 + z^{R_1} + z^{R_2} + \cdots + z^{R_{m-1}}), R_1 < R_2 < \cdots < R_{m-1}, m = 2,3, \cdots,$$
then there there exists an invertible Lebesgue  measure preserving transformation on the real line with simple Lebesgue spectrum \cite{Bourgain}, {\cite{Guenais}}, thus answering a question of Banach mentioned in the Scottish book. \\

The purpose of this note is to formulate and prove a necessary condition for the existence of a sequence of  flat polynomials in the above sense with real non-negative coefficients. A careful look at this condition shows that the problem of existence of an a.e. flat sequence of polynomials from the class $B$ is related to questions in combinatorial number theory (see section 7).\\

\section{a.e. flat sequence of polynomials}

\begin{Def}\label{def1}
 Let $S^1$ denote the circle group and let $dz$ denote the normalized Lebesgue measure on it. A sequence $P_j, j=1,2,\cdots$  of analytic trigonometric polynomials with $L^2(S^1,dz)$ norms 1 and their constant terms positive, is  said to be flat a.e. or a.e flat  if $|P_j(z)|\rightarrow  1$ a.e. $(dz)$ as $j \rightarrow \infty$. \\
\end{Def}

The sequences $P_j(z) =1, j=1,2,\cdots$ or $P_j = \sqrt{1 -\frac{1}{j}} + \sqrt{\frac{1}{j}}z, j=1,2,\cdots$ are obviously flat a.e. It is easy to give flat a.e. sequence $P_j, j=1,2,\cdots$ of polynomials with non-negative coefficients, where the largest of the coefficients of $P_j$ converges to 1. (Here and in the sequel, nonnegative will mean real and non-negative.) Next we observe  the following: If
$P_j  , j=1,2,$ is an a.e flat sequence of polynomials with non-negative coefficients and if $P_j(1) \rightarrow 1$ as $j \rightarrow \infty$, then the largest of the coefficients of
$P_j$ converges to 1 as $j\rightarrow \infty$. Indeed if for each $j$, $c_{k,j}, 0 \leq k\leq n_j$ are the coefficients of $P_j$, then
$$1 = \sum_{k=0}^{n_j}c_{k,j}^2 \leq \sum_{k=0}^{n_j}c_{k,j} \rightarrow 1 ~~~{\rm{as}}
~~~ j \rightarrow \infty,$$
from which it is easy to conclude that $\max_{0 \leq k\leq n_j}\{c_{k,j}\} \rightarrow 1$ as $j\rightarrow \infty$.\\

Let $P_j, j=1,2,\cdots$ be an a.e. $(dz)$ flat sequence with non-negative coefficients. Assume that
for a.e. $z$, $P_j(z) \rightarrow \phi(z)$, as $j \rightarrow \infty$, for some function $\phi$
of absolute  value 1 on $S^1$. Then $P_j, j=1,2,\cdots$ converges to $\phi$ weakly, whence Fourier coefficients of $\phi$ are all  non-negative. If $\phi$ has two or more coefficients positive
we can conclude that the constant function $1 = \phi {\overline{\phi}}$ has two or more
Fourier coefficients positive, which  is not true. Whence $\phi = z^n$ for some $n$, which
in turn implies that the largest coefficient of $P_j$ converges to 1 as $j \rightarrow \infty$.
In particular the simple minded way of constructing a.e. ($dz$) flat sequence of polynomials,
namely taking the partial sums of an analytic function on $S^1$  of absolute value 1 a.e. ($dz$),
will not yield such a sequence with non-negative coefficients and with maximum of the
coefficients uniformly bounded away from 1.\\

\section{Covariance matrix of $\big| P\big|^2$ and the quantity $C$}

    Consider a polynomial with non-negative coefficients of $L^2(S^1,dz)$ norm 1.
    Such a polynomial with $m$ non-zero coefficients can be written as:\\
 $$P(z) = \sqrt {p_0} + \sqrt {p_1}z^{R_1} + \cdots + \sqrt{p_{m-1}}z^{R_{m-1}},  ~~~\eqno (4)$$
 where each $p_i$ is positive and their sum is 1. Such a $P$ gives a probability measure $\mid P(z)\mid^2dz$ on the circle group which we denote by $\nu$. Now $\mid P(z)\mid^2$ can be written as
$$\mid P(z)\mid^2 = 1 + \sum_{\overset{k=-N,}{k\neq 0}}^{N}a_kz^{n_k} ,$$
where each $n_k$ is of the form $R_i - R_j,$ and each $a_k$ is a sum of terms of the type ${\sqrt{p_i}}{\sqrt{p_j}}, i\neq j$, with $R_j -R_i = n_k$ , $a_k = {a_{-k}}, 1\leq k \leq N$.
We will write $$L = \sum_{\overset{k=-N,}{k\neq 0}}^{N}a_k =  | P(1)|^2 -1.$$\\

 Then
 $$L = \sum_{\overset{0 \leq i,j\leq m-1,}{ i\neq j}}\sqrt {p_i}\sqrt{p_j},$$ is a function of
  probability vectors $(p_0, p_1, p_2, \cdots p_{m-1})$, which  attains its maximum value when each $p_i = \frac{1}{m}$, and the maximum value is $\frac{m(m-1)}{m} = {m-1}$.\\
  
  We conclude therefore that $\mid L \mid \leq m-1$. 
   We also note that $m - 1 \leq N \leq \frac{1}{2}m(m-1)$. So, when $p_i$'s are all equal
   and = $\frac{1}{m}$ we have 
   $$\frac{N}{L^2} \leq  \frac{m}{2(m-1)} \leq 1 \textrm{~~for~~} m \geq 2.$$

  For each $k, -N \leq k \leq N, k\neq 0$, let $D_k$ denote the cardinality of the set of pairs
  $(i,j), i \neq j , -N\leq i,j \leq N, i,j \neq 0$, such  that $n_j - n_i = n_k$. For each $k$, $D_k \leq 2N - 2\mid k\mid +2 \leq m(m-1)$, whence
  $$\Big| \sum_{\overset{k=-N}{k\neq 0}}^Na_k D_k\Big| \leq m(m-1) \sum_{\overset{k=-N}{k\neq 0}}^N a_k < m^3.$$

 We write
 $$A(P) = A = \sum_{\overset{k=-N}{k\neq 0}}^N a_k D_k,$$
 $$ B(P) =B = \sum_{\overset{-N \leq i,j\leq N}{0\neq i,j}}a_ia_j =
 \Big(|P(1)|^2 -1\Big)^2. $$

Consider the random variables  $X(k) =z^{n_k} - a_k$ with respect to the measure $\nu = \mid P(z)\mid^2dz$. We write $m(k,l) = \int_{S^1}X(k)\overline{X(l)}d\nu$, $-N \leq k,l\leq N, k, l \neq 0$ and $M$ for the correlation matrix with entries $m(k,l), -N \leq k,l\leq N, k, l \neq 0$. We call $M$ the covariance matrix associated to $\mid P(z)\mid^2$. Since linear combination of $X(k), -N \leq k\leq N, k\neq 0$, can vanish at no more than a finite set in $S^1$, and,
$\nu$ is non discrete, the random variables $X(k), -N \leq k \leq N, k\neq 0$ are linearly independent, whence the covariance matrix $M$ is non-singular.\\

Note that

$$m_{i,j} =\int_{S^1}z^{n_i - n_j}d\nu  - a_ia_j,~~~~ m_{i,i} = 1 -  a_i^2$$

Let $r(P) = r$ denote the sum of the entries of the matrix $M$. We have
\begin{eqnarray*}
r  &=&\sum_{\overset{k= -N}{k\neq 0}}^N\sum_{\{{i,j}, n_i - n_j = n_k ,  {i,j} \neq 0\}}m_{i,j}
  + \sum_{\overset{k= -N}{k\neq 0}}^Nm_{k,k}\\
&=&\sum_{\overset{k= -N}{k\neq 0}}^N\sum_{\{{i,j}, n_i - n_j = n_k ,  {i,j} \neq 0\}}(a_k - a_ia_j) + 2N - \sum_{\overset{i=-N}{i\neq 0}}^N a_i^2\\
&=&\sum_{\overset{k= -N}{k\neq 0}}^N a_kD_k + 2N - \sum_{\overset{-N\leq i,j\leq N}{i,j \neq 0}}a_ia_j\\
&=& A + 2N -  L^2.
\end{eqnarray*}

Since $A$ is of order at most $m^3$,  $N \leq \frac{1}{2}m(m-1)$, and $ L^2$ is of order $m^2$, we see that $r$ is of order at most $m^3$. We also note that the quantity
$C(P) = C = \sum_{\{(i,j), -N\leq i,j\leq N, i,j \neq 0\}}\mid m_{i,j}\mid$ is also of order at most $m^3$. Indeed
$$C  \leq \sum_{\overset{k=-N}{k\neq 0}}^N\Big(D_k a_k + \sum_{\{(i,j):i-j =k, i,j\neq 0\}}  a_ia_j\Big)+ 2N,$$
which shows that $C$ is of order at most $m^3$. \\

\section{Dissociated polynomials and generalized Riesz products}

 We say that a set of trigonometric polynomials is dissociated if in the formal expansion of product of any finitely many of them, the powers of $z$ in the non-zero terms are all distinct \cite{Abd-Nad}.\\

If $P(z) = \ds \sum_{j=-m}^m a_jz^j, Q(z) = \ds \sum_{j=-n}^{n}b_jz^j$, $m \leq n$, are two trigonometric  polynomials then
for some $N$, $P(z)$ and $Q(z^N)$ are dissociated. Indeed
$$P(z)\cdot Q(z^N) = \sum_{i=-m}^m\sum_{j=-n}^n a_ib_jz^{i+Nj}.$$
If we choose $N > 2n$, then we will have two exponents, say $i +Nj$ and $u+Nv$, equal if and only if $i-u = N(v-j)$ and since  $N$ is bigger than $2n$, this can happen if and only if $i=u$ and $j=v$. More generally, given any sequence $P_1, P_2, \cdots$ of polynomials one can find integers $1 = N_1 < N_2 < N_3 < \cdots,$  such that $P_1(z^{N_1}), P_2(z^{N_2}), P(z^{N_3}), \cdots$ are dissociated. Note that since the map $z \longmapsto z^{N_i}$ is measure preserving, for any $p > 0$  the $L^p(S^1, dz)$ norm of $P_i(z)$ and $P_i(z^{N_i})$ remain the same.\\

Now let $P_1, P_2,\cdots$ be a sequence of polynomials, each $P_i$ being  of $L^2(S^1, dz)$ norm 1. Then the constant term of each $\mid P_i(z)\mid^2$ is 1. If we choose  $1 = N_1 < N_2 < N_3 \cdots$  so that $\mid P_1(z^{N_1})\mid^2, \mid P_2(z^{N_2})\mid^2, \mid P_3(z^{N_3})\mid^2, \cdots$ are dissociated, then the constant term of each finite product
$$\prod_{j=1}^n\mid P_j(z^{N_j})\mid^2$$ is one so that each finite product  integrates to 1 with respect to $dz$. Also, since $\mid P_j(z^{N_j}) \mid^2, j =1,2, \cdots$ are dissociated,
for any given $k$, the $k$-th Fourier coefficient of $\prod_{j=1}^n\mid P_j(z^{N_j})\mid^2$ is either zero for all $n$, or, if it is non-zero for some $n = n_0$ (say), then its remains the same  for all $n \geq n_0$. Thus the measures $(\prod_{j=1}^n| P_j(z^{N_j})|^2)dz, n=1,2,\cdots$ admit a weak limit
on $S^1$. It is called the generalized Riesz product of the polynomials
$\mid P_j(z^{N_j})\mid^2, j=1,2,\cdots$.\cite{Host-Mela-Parreau}, \cite{Abd-Nad}.   Let $\mu$ denote this measure. It is known \cite{Abd-Nad} that
$ \prod_{j=1}^k |P_j(z^{N_j})|, k=1,2,\cdots$, converge in $L^1(S^1,dz)$ to
$\sqrt{\frac{d\mu}{dz}}$ as $k\rightarrow \infty$. It follows from this that if $\prod_{j=1}^k\mid P_j(z^{N_j})\mid , k=1,2,\cdots$ converge a.e. $(dz)$ to a finite positive value then $\mu$ has a part which is equivalent to Lebesgue measure.\\

\section{A necessary condition for a.e. flatness}

We will now consider a sequence $P_j(z),$ $j=1,2,\cdots$ of polynomials, each $P_j$ of $L^2(S^1, dz)$ norm 1, and non-negative coefficients. The quantities $A(P_j), C(P_j)$ etc will now written as $A_j, C_j$ etc. It will follow from our considerations below that {\it if a sequence of polynomials $P_j, j=1,2,\cdots$ from the class $B$ is flat then $\frac{C(P_j)}{m_j^2} \rightarrow \infty$ as $j \rightarrow \infty$}.\\

The main theorem is as follows:\\

 \begin{Th}\label{mainth}
  If $L_j, j=1,2,\cdots$ are uniformly bounded away from 0 and  $~~\ds \lim_{j\rightarrow \infty}|P_j(z)| = 1$ a.e. $(dz)$ then
 $\frac{C_j}{m^2_j}\rightarrow \infty$ as $j\rightarrow \infty$.
\end{Th}

To prove this we need the following lemma, which should not be viewed as new singularity result for  Riesz products, rather it is an ancillary result needed to prove the main theorem.\\

\begin{lem}\label{lem1}

  {If $P_j(z), j =1,2,\cdots$ is a sequence of analytic trigonometric polynomials of $L^2(S^1, dz)$ norm 1  such that\\
  (i) $L_j, j=1,2,\cdots$ are uniformly bounded away from 0,\\
 (ii) the polynomials $\mid P_j\mid^2, j =1,2,\cdots$ are dissociated\\
 (iii) $\sum_{j=1}^\infty\frac{L^2_j}{ {C_j}} = \infty$,\\
  then $\mu = \Pi_{j=1}^\infty \mid P_j(z)\mid^2$ is
singular to its translate $\mu_u$ for every $u\in S^1$ for which the  sequence $\mid P_j(u)\mid \rightarrow 1$, as $j\rightarrow \infty$. }\\
\end{lem}

\begin{proof}
 By Banach-Steinhaus theorem there exist $b_j, j=1,2,\cdots $, with their sum of absolute squares finite such that for each $j$, $\frac{L_j}{C_j}b_j \geq 0$ and $\sum_{j=1}^\infty \frac{L_j}{\sqrt {C_j}}b_j = \infty$. Fix a $v\in S^1$ such that  $\mid P_j(v)\mid \rightarrow 1$ as $j \rightarrow \infty$.
Note that
$$\sum_{j=1}^\infty \Big(\sum_{\overset{k=-N_j}{k\neq 0}}^{N_j}a_j\big(1  - v^{n_{k,j}}\big)\Big) =
\sum_{j=1}^\infty\Big(L_j - \big(\mid P_j(v)\mid^2 -1\big)\Big).$$
Since $\mid P_j(v)\mid^2\rightarrow 1$ as $j \rightarrow \infty$, the series
$\sum_{j=1}^\infty\Big(\frac{L_j-(\mid P_j(v)\mid^2 -1)}{\sqrt{C_j}}\Big)b_j$  diverges.
Let $B_j$ be the $1\times 2N_j$ matrix with all entries equal to $\frac{b_j}{\sqrt {C_j}}, j=1,2,\cdots$. Then

$$(M_jB_j,B_j) = \frac{r_j\mid b_j\mid^2}{C_j} \leq \mid b_j\mid^2,$$
whence $\sum_{j=1}^\infty (M_jB_j, B_j)$  is a finite sum, which in turn implies that the series in $j$\\
$$\sum_{j=1}^\infty\sum_{\overset{k=-N_j}{k\neq 0}}^{N_j} \frac{b_j}{\sqrt {C_j}}(z^{n_{k,j}} - a_{k,j})$$
 converges a.e. ($\mu$) over a subsequence.\\

Consider now the translated measure $\mu_v(\cdot) = \mu(v(\cdot))$. We have
$$\int_{S^1}z^{n_{k,j}}d\mu_v = v^{-n_{k,j}}a_{k,j}.$$
The covariance  matrix $M_{v,j}$ of the random variables
$z^{n_{k,j}} - v^{-n_{k,j}}a_{k,j}, -N_j \leq k \leq N_j, k\neq 0$ with respect to the
translated measure $\mu_v$ has entries $ v^{-(n_{k,j}- n_{l,j})}m_{k,l}$, which can be seen to
be unitarily equivalent to $M_j$. Indeed, $$M_{v,j} = U_j^{-1}M_jU_j,$$ where
  $U_j$ is a $2N_j\times 2N_j$ diagonal matrix with entries $$v^{n_{-N_j,j}}, v^{n_{-N_j +1,j}}, \cdots, v^{n_{-1,j}}, v^{n_{1,j}}\cdots, v^{n_{N_j-1,j}}, v^{n_{N_j, j}},$$
  along the diagonal in that order.\\

  We note that
  $$\sum_{j=1}^\infty(M_{v,j}B_j, B_j)$$
  $$= \sum_{j=1}^\infty \frac{r_{{v,j}}}{C_j}\mid b_j\mid^2 < \infty,$$
where $r_{v,j}$ is the sum of the entries of the of the matrix $M_{v,j}, j=1,2,\cdots$. It is clear that for all $j$,
$|r_{v,j}| \leq C_j$. \\

As before we conclude that the series
$$ \sum_{j=1}^\infty \Big(\sum_{\overset{k=-N_j,}{k\neq 0}}^{N_j}\frac{b_j}{\sqrt{C_j}}\big(z^{n_{k,j}} - v^{-n_{k,j}}a_{k,j}\big)\Big)$$
converges a.e $\mu_v$ over a subsequence.

If $\mu$ and $\mu_v$ are not mutually singular, then there exist an $z_0\in S^1$ and an
increasing sequence $K_p, p=1,2,\cdots$ of natural numbers such that the sequences

$$\sum_{j=1}^{K_p}\sum_{\overset{k=-N_j}{k\neq 0}}^{N_j}\frac{b_j}{\sqrt {C_j}}(z_0^{n_{k,j}} - a_{k,j})$$
$$\sum_{j=1}^{K_p}\sum_{\overset{k=-N_j}{k\neq 0}}^{N_j}\frac{b_j}{\sqrt{C_j}}(z_0^{n_{k,j}} - v^{-n_{k,j}}a_{k,j})$$

converge to a finite number as $p \rightarrow \infty$. The difference of the two partial sums is
$$\sum_{j=1}^{K_p}\sum_{\overset{k=-N_j}{k\neq 0}}^{N_j}\frac{b_j}{\sqrt{C_j}}a_{k,j}(1 - v^{-n_{k,j}}),$$
which diverges as $p\rightarrow \infty$. The contradiction shows that $\mu$ and $\mu_v$ are singular.\\
\end{proof}

 The following theorem is proved in \cite{Abd-Nad}.

\begin{Th}\label{th7}   Let $P_j, j =1,2,\cdots$ be a sequence of non-constant polynomials
of $L^2(S^1,dz)$ norm 1 such that $\lim_{j\rightarrow \infty}\mid P_j(z)\mid =1 $ a.e. $(dz)$ then there exists a subsequence $P_{j_k}, k=1,2,\cdots$ and natural numbers $l_1 < l_2 < \cdots$ such that the polynomials $P_{j_k}(z^{l_k}), k=1,2,\cdots $ are dissociated and
the infinite product $\prod_{k=1}^\infty |P_{j_k}(z^{l_k})|^2$ has finite nonzero value a.e $(dz)$.

\end{Th}

 We now prove Theorem \ref{mainth}.

\begin{proof}[\textbf{Proof of Theorem \ref{mainth}}] Under the hypothesis of the theorem, by theorem 5.3 we get a subsequence
$P_{j_k} = Q_k, k=1,2, \cdots$ and natural numbers $l_1 < l_2 < \cdots$ such that the polynomials
$\mid Q_k(z^{l_k})\mid^2, k=1,2,\cdots $ are dissociated and the infinite product
$\prod_{k=1}^\infty\mid Q_k(z^{l_k})\mid^2$ has finite non-zero limit a.e. $(dz)$. Also, since the absolute squared  $Q_k(z^{l_k})$'s are dissociated,  the measures $\mu_n \stackrel{\textrm{def}}{=} \prod_{k=1}^n\mid Q_k(z^{l_k})\mid^2dz$ converge weakly to a measure $\mu$ on $S^1$
for which $\frac{d\mu}{dz} > 0$ a.e $(dz)$, indeed
$$\frac{d\mu}{dz} = \prod_{k=1}^\infty\mid Q(z^{l_k})\mid^2 ~~{\rm{a.e.}}~~~~(dz).$$
 Since the map $z \longmapsto z^{l_k}$ preserves the Lebesgue measure on $S^1$,  the $m_{j_k}(u,v)$'s for  $\mid P_{j_k}(z^{l_k})\mid^2dz$ remains the same as for  $\mid P_{j_k}(z)\mid^2dz$. If $\sum_{k=1}^\infty \frac{L^2_{j_k}}{C_{j_k}} = \infty$, then by Lemma 5.2 $\mu$ will be singular to $\mu_u$ for a.e. $u$. This is false since $\frac{d\mu}{dz} > 0$ a.e. ($dz$). So $\sum_{k=1}^\infty {\frac{L^2_{j_k}}{C_{j_k}}} < \infty$.  If $\frac{L^2_{j}}{C_{j}}, j=1,2,\cdots$ does not tend to $0 $ as $j\rightarrow \infty$ then over a subsequence these ratios remain bounded away from $0$. But by the  above considerations, over a further subsequence these ratios have a finite sum, which is a contradiction. So $\frac{L^2_j}{C_j} \rightarrow 0$ as $j\rightarrow \infty$.
 \end{proof}

 Note that if $P_j, j =1,2,\cdots$ is a an a.e. flat sequence of polynomials from the class
 $B$, then $L_j = m_j-1,  j=1,2,\cdots$ is bounded away from zero, we see that $\frac{C(P_j)}{L(P_j)^2} \rightarrow \infty$ as $j \rightarrow \infty$.\\

\section{Connection with combinatorial number theory }

In this section we discuss the ratios  $\frac{C}{m^2}$ for the class $B$. In particular we give a sequence $P_j, j=1,2,\cdots$ from this class for which $\frac{C(P_j)}{m_j^2}, j =1,2, \cdots$ diverges but $P_j, j =1,2,\cdots$ is  not flat in a.e. ($dz$) sense.\\

 For a given polynomial $P(z) = \frac{1}{\sqrt m}( 1 + z^{R_1} + z^{R_2} + \cdots + z^{R_{m-1}})$ of class $B$,
with
$$\mid P(z) \mid^2 = 1 +  \sum_{j=1}^Na_j(z^{n_j} + z^{-n_j}),$$
we know that $\frac{C(P)}{m^2}$ has the same order as $ \frac{2\sum_{j=1}^Na_jD_j}{m^2}$.
However just ensuring that each $D_j$ receives maximum possible value, namely $N-j$, is not enough to ensure that $2\sum_{j=1}^Na_jD_j$ is large in comparison with ${m^2}$.  For consider
the case when for each $j$, $R_j = j$, so that
$$P(z) = \frac{1}{\sqrt m}(1 +z +z^2 + \cdots + z^{m-1})$$
$$\mid P(z)\mid^2 = 1 + \frac{1}{m}\sum_{j=1}^{m-1}(m -j)(z^j+ z^{-j}).$$
 Now
each $D_j = m-j$, so that
$$2\sum_{j=1}^{m-1}a_j D_j = 2\frac{1}{m}\sum_{j=1}^{m-1}(m-j)^2=
2\frac1{m}\sum_{j=1}^{m-1}j^2=\frac{(m+1)(2m+1)}{3} $$
which is of order $m^2$.\\

One can ensure $C$ large in comparison with $m^2$ if each $D_j$ has its maximum possible value, namely, $N-j$, and $N$ is of  higher order than $m$. Using some combinatorial number theory one can arrange this.\\

 Let $R$ be a natural number $> 2$ and let $m\geq 2$ be a natural number $\leq R$. Write $R_0 =0.$
 Let $R_0 < R_1 < R_2 < \cdots < R_{m-1} = R$ be a set of $m$ integers between 0 and $R$. Denote it by $S$. Note that $0$ and $R$ are in $S$. Let
      $$P_S(z) = \frac{1}{\sqrt{\mid S\mid}} \sum_{j \in S}z^j,$$

 $$\mid P_S(z)\mid^2 = 1 + \frac{1}{\mid S\mid}\sum_{j =1}^Nd_j(z^{n_j} + z^{-n_j})$$
where for each $j$, $d_j =$ number of pairs $(a, b), a,b \in S, b-a = j$.
 Let
 $$(S - S)^+ = \{b-a: a,b \in S, a < b\} = \{n_1 < n_2 < \cdots < n_N = R\},$$
which is the set of positive differences of elements in $S$.
   If $d_j =1$ for all $j$, then $S$ is called Sidon subset of $[0,R]$.
  As pointed out to us by R. Balasubramanian, if $S$ is a Sidon set $\subset [0, R]$, then
  $(S-S)^+ \neq [1, R]$
  (unless $R \leq 6$.). This is a consequence of a well known result of Erd\"{o}s and
  Turan \cite{erd}  which says
 that if $S \subset [1,R]$ is a Sidon set then $\mid S\mid $  is at most
 $R^{\frac{1}{2}} +  R^{\frac{1}{4}} + 1$, see \cite{lind}. So, if $S$ is a Sidon set then the
 cardinality $(S - S)^+$
is $\frac{1}{2}\mid S\mid  (\mid S\mid - 1) <  R$.\\

Let $M(S) = \max\{d_j:1 \leq j \leq N\}$. The quantities $M(S)$ and $\mid (S - S)^+\mid = N$ are
in some sense `balanced' in that if one is large the other is small, and $M(S)\mid (S-S)^+\mid$
seems to be of order $\mid S\mid^2$. Obviously, This is true when $S$ is a Sidon set and the
other extreme case when $S = [0, m-1]$ \\

We do not know if one can choose, for each $R$, a suitable Sidon set
$ S_R \subset [0, R]$, with $0, R \in S_R$, such that ratios
$\frac{C(P_{ S_R})}{\mid S_R\mid^2}, R =1, 2, \cdots$ are unbounded,
where $P_{ S_R}$ is the polynomial in class $B$ with frequencies in
$S_R$, and additionally, if such a sequence of polynomials  can be  flat in a.e. ($dz$) sense.\\     \\

Let
$$ \lambda (R) = \min \Big\{\mid S\mid : S\subset [0,R], (S - S)^+ = [1, R]\Big\}.$$

For simplicity we discuss  $\lambda (R^2)$  rather than $\lambda (R)$.\\
We have
$$\sqrt{2} R < \lambda (R^2) \leq  2R.$$

To see the left hand side of this inequality note that
$$\frac{1}{2}\sqrt{2} R(\sqrt{2} R -1) = R^2 -\frac{1}{\sqrt {2}}R < R^2,$$
while the right hand side follows from the observation that the set
$$S = [0, R-1] \cup \{R, 2R, 3R, \cdots, (R-1)R, R^2\}$$
has  $2R$ elements and $(S-S)^+ = [1, R^2]$\\

We now show that $\frac{C}{m^2}$ is not bounded over the class $B$. For a given positive
integer
$R > 2$ choose $S\subset [0,R^2]$ of cardinality $\lambda (R^2)$ and such that
$(S - S)^+ = [1, R^2]$. Let $m$ denote $\lambda (R^2)$, let $R_0< R_1 < \cdots < R_{m-1} = R^2$
 be the set $S$. Let
 $$P(z) = \frac{1}{\sqrt m}( 1 + z^{R_1} + z^{R_2} + \cdots + z^{R_{m-1}})$$

$$\mid P(z)\mid^2 = 1 +\frac{1}{m}\sum_{j=1}^{R^2}d_j(z^{j} + z^{-j})$$

Now $$C(P) \geq A(P) = 2\sum_{j=1}^{R^2}\frac{1}{m}d_jD_j > 2\sum_{j=1}^{R^2}\frac{1}{m}D_j$$
$$= 2\sum_{j=1}^{R^2}\frac{1}{m}(R^2-j) = \frac{1}{m}(R^2-1)R^2$$
$$\geq \frac{1}{2}(R^2-1)R,$$
since $m = \lambda (R^2) \leq 2 R $. Hence  $\frac{C}{m^2}$ is unbounded over the class $B$.\\

We now give an example of a sequence $P_j, j=1,2,\cdots$ from the class $B$ for which
$\frac{C(P_j)}{m_j^2}\rightarrow \infty$ but the sequence $P_j, j=1,2,\dots$ is not flat in
a.e ($dz$) sense.\\
Let
$$P_j(z) = \frac{1}{\sqrt {2j}}\Big( \sum_{i=0}^{j-1}z^i + \sum_{i=1}^jz^{ij}\Big) = \frac{1}{\sqrt {2j}}\frac{1 - z^{j}}{1-z} + \frac{1}{\sqrt{2j}}\frac{1-z^{j^2}}{1-z^j},$$
then clearly, for a given  $z \neq  1$, $P_j(z) \rightarrow 0$ over every subsequence $j_n, n=1,2, \cdots$ over which $z^{j_n}, n= 1,2 \cdots$ stays uniformly  away from 1,
whence $P_j(z), j =1,2,\cdots$ is not a flat sequence in a.e. ($dz$) sense.

   Note that  $\mid S_j\mid = 2j$ and   $\mid P_j(z)\mid^2$ admits
  all the frequencies from 1 to $ j^2$,  whence, as seen above, $\frac{C(P_j)}{j^2} \rightarrow
\infty$ as $j \rightarrow \infty$.\\

 Since  $\frac{1}{2}(\sqrt{2}R +1)\sqrt{2}R = R^2 + \frac{1}{\sqrt{2}}R > R^2$, it may seems
 natural to surmise that $\lambda (R^2) < \sqrt 2 R + K$ for some fixed constant $K$
 independent of $R$. However, as shown to us by A. Ruzsa, this is
 false.\\
 Indeed there is a constant $c > \sqrt 2$ such that $ c R \leq \lambda (R^2)$, as shown below.
 Let $\phi (R) = \frac{\lambda (R^2) - \sqrt{2} R}{R} $. We show that $\phi (R)$ is
 uniformly bounded away from zero over all $R$. If not, $\phi (R)$ will converge to zero over a
 subsequence of natural numbers. Without loss of
  generality we assume that $\phi (R) \rightarrow 0$ as $R \rightarrow \infty$.
For each $R$, let $S_R =S$ be a subset of $[0, R^2]$ of cardinality $\lambda (R^2)$ such that
$(S-S)^+ = [1, R^2]$. Consider
\begin{eqnarray*}
0 &\leq& \Big| \sum_{j \in S(R)}z^j\Big|^2\\
&=&\sum_{j = -R^2}^{R^2}z^j + \sum_{j= -R^2}^{R^2}(d_{j}-1)z^j\\
&<& \sum_{j =-R^2}^{R^2}z^j + 4\phi(R) R^2,
\end{eqnarray*}
since $\mid (S-S)^+\mid < R^2 + 4\phi (R)R^2$ for large $R$. Put $z = e^{iv}$. We get
$$0 \leq \frac{\sin(R^2 + \frac{1}{2})v}{\sin\frac{1}{2}v} + 4\phi (R)R^2$$
which is a contradiction since the right hand side takes negative values for large $R$ and suitable $v$, e.g,
for $v = \frac{3\pi}{2 (R^2 +\frac{1}{2})}$. Whence $\phi (R)$ is bounded away from 0 uniformly in $R$.\\

We give below some probabilistic considerations which need further investigation. Let
$R > 2$ be an integer, and let $S\subset [0, R^2]$ of cardinality $2R$, with $0,R \in S$. Let $\Omega_R$
denote the the collection of all such subsets $S$ in $[0,R^2]$. Cardinality of $\Omega_R$ is  $(^{R^2-1}_{2R-2})$. Equip $\Omega$ with uniform distribution, denoted by $\mathbb{P}_R.$ Let $P(R, S)$ denote the polynomial of class $B$ with frequencies in $S$. For a fixed $\epsilon > 0$, one can consider  $ E(\epsilon, R) = \mathbb{P}_R(\{S: \mid\mid(\mid P(R,S)\mid^2 - 1)\mid\mid_1 > \epsilon\})$. If for every $\epsilon > 0$, $E(\epsilon, R) \rightarrow 0$ as $R \rightarrow \infty$, we will have a probabilistic proof of the existence of a sequence flat polynomials
(in a.e. ($dz$) sense) in the class $B$.\\

For more on flat polynomials, not necessarily with non-negative coefficients, see
\cite{Abd-Nad1}.

{\bf{Acknowledgement.}} M. G. Nadkarni would like to thank University of Rouen for a month long visiting appointment during which the paper was revised and completed.


\begin{thebibliography}{}
%
%
%
\bibitem{Abd-Nad}
E. H. ~el Abdalaoui and M. Nadkarni, \emph{ Calculus of Generalized Riesz Products ,}
Contemporary Mathematics(AMS) 631 (2014), pp. 145-180.

\bibitem{Abd-Nad1}
E. H. ~el Abdalaoui and M. Nadkarni,{ Some notes on flat polynomials,} Arxiv
\bibitem {Bourgain}

J. ~Bourgain,{ On the spectral type of Ornstein class one transformations, }{\it Isr. J. Math. },{\bf 84
}(1993), 53-63.
\bibitem{erd}
P. Erd\"{o}s and P. Turan, {\em On a Problem of Sidon in Additive Number Theory and
and some related problems} J. London. Math. Soc. {\bf 16}(1941) 212-215.


\bibitem{Guenais}
M. ~Guenais, {\em Morse cocycles and simple Lebesgue spectrum, } Ergodic Theory Dynam. Systems,
 \textbf{19} (1999), no. 2, 437-446.

\bibitem{Host-Mela-Parreau}
 B. Host, J.-F. M\'ela, F. Parreau, {\em Non-singular transformations and spectral analysis of measures}, Bull. Soc. math. France {\textbf{119}} (1991), 33-90.

\bibitem{lind}
B. Lindstro\"{o}m, {\em An Inequality for $B_2$ Sequences,} J. Comb. Th., {\bf 6} 211-212 (1969).
\bibitem{Peyriere}
J. Peyri\`ere, {\it \'Etude de quelques propri\'et\'es des produits de Riesz}, Ann.\ Inst.\ Fourier,
Grenoble \textbf{25}, 2 (1975), 127--169.
\end{thebibliography}
\end{document}